\documentclass[12pt,a4paper]{article}
\usepackage{theorem}%
\usepackage{graphicx}%
\usepackage{amssymb}%
\newtheorem{lemma}{Lemma}

\newtheorem{theorem}[lemma]{Theorem}
\newtheorem{corollary}[lemma]{Corollary}
{\theorembodyfont{\upshape}}
{\theorembodyfont{\upshape}}
{\theorembodyfont{\upshape}}
{\theorembodyfont{\upshape}\newtheorem{example}[lemma]{Example}}
{\theorembodyfont{\upshape}}

\newcommand{\Z}{{\bf Z}}
\newcommand{\R}{{\bf R}}
\renewcommand{\H}{{\bf H}}
\newcommand{\C}{{\bf C}}

\newcommand{\rme}{{\rm e}}
\newcommand{\rmd}{{\rm d}}
\newcommand{\cA}{{\cal A}}
\newcommand{\cB}{{\cal B}}
\newcommand{\cC}{{\cal C}}
\newcommand{\cD}{{\cal D}}

\newcommand{\cF}{{\cal F}}

\newcommand{\cH}{{\cal H}}
\newcommand{\cI}{{\cal I}}
\newcommand{\cJ}{{\cal J}}
\newcommand{\cK}{{\cal K}}
\newcommand{\cL}{{\cal L}}

\newcommand{\cS}{{\cal S}}
\newcommand{\cU}{{\cal U}}

\newcommand{\cV}{{\cal V}}
\newcommand{\cY}{{\cal Y}}
\newcommand{\sig}{\sigma}
\newcommand{\alp}{\alpha}
\newcommand{\bet}{\beta}
\newcommand{\gam}{\gamma}

\newcommand{\lam}{\lambda}
\newcommand{\del}{\delta}
\newcommand{\eps}{\varepsilon}

\newcommand{\lap}{{\Delta}}

\newcommand{\Dom}{{\rm Dom}}

\newcommand{\norm}{\Vert}
\renewcommand{\Re}{{\rm Re}\,}

\newenvironment{proof}{\textit{Proof.}}{\hspace*{\fill}$\Box$}%
\newcommand{\Note}{\textit{Note.}\hspace{0.5em}}

\newcommand{\Schrodinger}{Schr\"odinger }

\newcommand{\ca}{$C^\ast$-algebra}
\newcommand{\csa}{$C^\ast$-subalgebra}
\newcommand{\sigess}{\sig_{\rm ess}}
\renewcommand{\emptyset}{\varnothing}
\newcommand{\ol}[1]{\overline{#1}}
\newcommand{\ti}[1]{\widetilde{#1}}
\newcommand{\ha}[1]{\widehat{#1}}
\newcommand{\pd}{\partial}
\newcommand{\sub}{\subseteq}
\newcommand{\ops}{one-parameter semigroup}
\newcommand{\implies}{\Rightarrow}
\newenvironment{choices}{\left\{\begin{array}{ll}}{\end{array}\right.}%
\newenvironment{eq}{\begin{equation}}{\end{equation}}%
\usepackage{hyperref}
\title{Decomposing the Essential Spectrum}
\author{E B Davies}
\date{18 September 2008}
\begin{document}
\maketitle

\begin{abstract}
We use \ca\ theory to provide a new method of decomposing the essential
spectra of self-adjoint and non-self-adjoint Schr\"odinger operators
in one or more space dimensions.

MSC-class: 35P05, 81Qxx, 47C15, 47Lxx.

keywords: essential spectrum, \ca, non-self-adjoint, discrete
Schr\"odinger operator
\end{abstract}

\section{Introduction\label{sectA}}

In a recent study by Hinchcliffe \cite{Hinchcliffe1} of the spectrum
of a periodic, discrete, non-self-adjoint Schr\"odinger operator
acting on $\Z^2$ with a dislocation along $\{0\}\times \Z$, we were
struck by the fact that the essential spectrum of the operator,
defined by means of the Calkin algebra, divides into two parts, one
of which occupies a region in the complex plane, the other being one
or more simple curves; the curves are associated with surface states
confined to a neighbourhood of the dislocation. The same phenomenon
occurs in the self-adjoint case, but here the distinction is between
parts of the (real) spectrum that have infinite spectral
multiplicity and other parts with finite multiplicity, at least in
two dimensions.

In this paper we describe a new method of decomposing the essential spectrum of a self-adjoint or non-self-adjoint \Schrodinger operator into parts by using the two-sided ideals of a certain standard \ca. Our conclusion is that one can define different types of essential spectrum, provided one is given this extra structure; we warn the reader that the spectral classification that we obtain is not a unitary invariant of the operators concerned. However, the \ca\ used is the same for all the applications considered so the results obtained have a high degree of model-independence.

Some of our spectral results can be proved by methods that are geometric in the sense that they involve Hilbert space methods rather than \ca s. An  advantage of the approach described here is that instead of dealing with new applications by invoking analogy and experience, the use of \ca s enables one to formulate simple general theorems that cover applications directly. The method accommodates many of the technical hypotheses that have been used in the field within a single formalism.

In Sections~\ref{sectB} and \ref{sectC} we investigate the relevant \ca\ theory without reference to its application. Section~\ref{sectD} is devoted to showing how to apply the results to discrete \Schrodinger operators. Theorems~\ref{line} and \ref{cross} describe the spectrum when a periodic potential has a dislocation on one or both of the two axes in $\Z^2$; the second possibility has not previously been considered. After a substantial amount of preparatory work, we turn in Section~\ref{sectE} to the study of \Schrodinger and more general differential operators acting in $L^2(\R^d)$, and show that the abstract methods developed earlier can be applied to their resolvent operators under suitable hypotheses. The spectral mapping theorem then allows one to pull the results back to the original operators. Example~\ref{mbso} explains the application of the methods to multi-body \Schrodinger operators. Finally, in Section~\ref{hyperb} we show that our methods are not only relevant in a Euclidean context. We prove that the $C^\ast$-algebraic assumptions are satisfied when considering the Laplace-Beltrami operator on three-dimensional hyperbolic space by writing down the explicit formulae available in this case; the same applies to a wide variety of other Riemannian manifolds but more general methods are needed.

It might be thought that \ca\ methods cannot cope with problems in scattering theory because the relevant unitary groups $\rme^{-iHt}$ are strongly but not norm continuous. In fact a large part of the `geometric' approach to scattering theory as developed by Enss and others depends on considering the large time asymptotics of $\rme^{-iHt}A$ where the `localization operator' $A$ imposes position and momentum (or energy) cut-offs. The theory depends on making suitable norm estimates and the relevant integrals are mostly norm convergent or can be rewritten as norm convergent integrals; see \cite{enss,simsc,EBDsc}. In this context the important point is that a uniformly bounded family of operators $B_t$ depends strongly continuously on the parameter $t$ if and only if $B_tA$ depends norm continuously on $t$ for enough compact operators $A$; the union of the ranges of the $A$ chosen must span a dense linear subspace of $\cH$. The trick is to choose $A$ appropriately, and this may be done in several ways.

\section{Some \ca\ theory\label{sectB}}

Throughout this section $\cA$ will denote a (usually
non-commutative) \ca\ with identity, and $\cJ$ will denote a
(closed, two-sided) ideal in $\cA$. It is well-known that such an
ideal is necessarily closed under adjoints and that $\cA/\cJ$ is
again a \ca\ with respect to the quotient norm. See
\cite[Chapter~1]{Dix} or \cite[Chapter~1]{ped} for various standard
facts about \ca s that we will use without further comment.

If $x\in\cA$ then we denote the spectrum of $x$ by $\sig(x)$; it is
known that if $\cA$ is replaced by a larger \ca, $\sig(x)$ does not
change. If $\cJ$ is an ideal in $\cA$ we denote the natural map of
$\cA$ onto the quotient algebra $\cA/\cJ$ by $\pi_\cJ$. If several
ideals $\cJ_r$ are labelled by a parameter $r$, we write $\pi_r$
instead of $\pi_{\cJ_r}$ for brevity, and also put
$\sig_r(x)=\sig(\pi_{\cJ_r}(x))$

\begin{lemma}\label{kideals}
If the ideals $\{\cJ_r\}_{r=1}^k$ in $\cA$ satisfy
\[
\cJ_k\subseteq \cJ_{k-1}\subseteq \ldots \subseteq \cJ_2\subseteq
\cJ_1\subseteq \cA
\]
then
\[
\sig_1(x)\subseteq\sig_2(x)\subseteq\ldots\subseteq\sig_{k-1}(x)%
\subseteq\sig_k(x)\subseteq\sig(x).
\]
\end{lemma}
\begin{proof} We first put $\cJ_{k+1}=\{0\}$, so that $\cA/{\cJ_{k+1}}=\cA$
and $\sig_{k+1}(x)=\sig(x)$. Suppose that $1\leq r\leq s\leq k+1$
and that $\lam\notin \sig_s(x)$. Then there exists $y\in \cA$ such
that
\[
(\pi_s(x)-\lam 1)\pi_s(y)=\pi_s(y)(\lam 1-\pi_s(x))=1
\]
in $\cA/\cJ_s$. Hence there exist $u,v\in \cJ_s$ such that
\[
(x-\lam 1)y=1+u, \hspace{3em} y(\lam 1-x)=1+v.
\]
Applying $\pi_r$ to both equations and using the fact that
$\cJ_s\subseteq \cJ_r$ we obtain
\[
(\pi_r(x)-\lam 1)\pi_r(y)=\pi_r(y)(\lam 1-\pi_r(x))=1.
\]
Hence $\lam\notin \sig_r(x)$ and $\sig_r(x)\subseteq \sig_s(x)$.
\end{proof}

\Note If $\cA=\cL(\cH)$ and $\cJ$ is the ideal $\cK(\cH)$ of all
compact operators on the Hilbert space $\cH$, then
$\sig(\pi_\cJ(x))$ is (one of several inequivalent definitions of)
the essential spectrum of $x$ by \cite[Th. 4.3.7]{LOTS}. Needless to
say we are interested in more general examples.

There are several ways of constructing $\cA$ and the relevant ideals
$\cJ_r$. Given $\cJ\!$, the largest choice of $\cA$ is described in
(\ref{Adef1}) and more concretely in Lemma~\ref{Adef3}. If one
wishes make another choice, call it $\tilde{\cA}$, one has to
confirm that $\cJ\subseteq \tilde{\cA}\subseteq \cA$.

\begin{theorem}\label{Jstart} Let $\cB$ be a \ca\ with identity and let $\{p_n\}_{n=1}^\infty$ be an increasing sequence of orthogonal projections in $\cB$ with $p_n\not= 1$ for every $n$. Then the norm closure $\cJ$ of
\[
\cJ_0=\{ x\in \cB: \exists n\geq 1.\,\, p_nx=xp_n=x\}
\]
is a \csa\ that does not contain the identity of $\cB$. We have
\begin{equation}
\cJ=\{ x\in\cB:\lim_{n\to\infty}\norm x-p_nxp_n\norm
=0\}.\label{J0closure}
\end{equation}
Moreover $\cJ$ is an ideal in the \ca\ with identity $\cA$ defined
by
\begin{equation}
\cA=\{ a\in \cB:a\cJ\subseteq \cJ \mbox{ {\rm and} } \cJ a\subseteq
\cJ\}.\label{Adef1}
\end{equation}
If $\cB=\cL(\cH)$ and $p_n$ converge strongly to $I$ as $n\to\infty$
then
\[
\cK(\cH)\subseteq \cJ\subseteq \cA,
\]
so
\[
\sig(\pi_\cJ(x))\subseteq \sigess(x)\subseteq \sig(x)
\]
for all $x\in\cA$.
\end{theorem}

\begin{proof} First note that if $p_nx=xp_n=x$ then $p_mx=xp_m=x$ for all
$m\geq n$. It follows by elementary algebra that $\cJ_0$ is a
$^\ast$-subalgebra of $\cB$, and this implies the same for $\cJ$. If
$x\in\cJ_0$ then there exists $n$ for which $1-p_n=(1-p_n)(1-x)$.
Therefore
\[
1=\norm 1-p_n\norm=\norm (1-p_n)(1-x)\norm\leq \norm 1-p_n\norm\norm
1-x\norm =\norm 1-x \norm\
\]
because $p_n\not= 1$ for every $n$. Hence $\norm 1-x\norm \geq 1$
for all $x\in\cJ$ and we can deduce that $1\notin \cJ$.

If $x\in\cB$ and $\lim_{n\to\infty}\norm x-p_nxp_n\norm =0$ then the
fact that $p_nxp_n\in\cJ_0$ implies that $x\in\cJ$. Conversely if
$x\in\cJ$ and $\eps>0$ then there exists $y\in\cJ_0$ such that
$\norm x-y\norm <\eps$. There now exists $N\geq 1$ such that
$y=p_nyp_n$ for all $n\geq N$. For all such $n$ we have
\begin{eqnarray*}
\norm x-p_nxp_n\norm &\leq& \norm x-y\norm+\norm y-p_nxp_n\norm\\
&=& \norm x-y\norm+\norm p_nyp_n-p_nxp_n\norm\\
&\leq & 2\norm x-y\norm \\
&<& 2\eps.
\end{eqnarray*}
Hence $\lim_{n\to\infty}\norm x-p_nxp_n\norm =0$.

The proofs that $\cA$ is a \ca\ with identity and that $\cJ$ is an
ideal in $\cA$ are both elementary algebra.

If $\cB=\cL(\cH)$ then in order to prove that $\cK(\cH)\subseteq
\cJ$ it is sufficient by (\ref{J0closure}) and a density argument to
observe that if $x$ is a finite rank operator then
$\lim_{n\to\infty}\norm x-p_nxp_n\norm \to 0$. The final inclusion
of the theorem follows from Lemma~\ref{kideals}.
\end{proof}

The following provides an alternative description of $\cA$.

\begin{lemma}\label{Adef3} Let $\cB$, $\{p_n\}_{n=1}^\infty$, $\cJ$ and $\cA$ be defined as in Theorem~\ref{Jstart}. Let
\begin{equation}
\cD_0=\{ a\in\cB:\forall n\geq 1.\,\exists m\geq n.\,p_m a
p_n=ap_n.\}\label{cD0def}
\end{equation}
and
\begin{equation}
\cD=\{ a\in\cB:\forall n\geq 1.\, ap_n\in\cJ\}.\label{cDdef}
\end{equation}
Then $\cD_0\subseteq\cD$ and $\cA=\cD\cap\cD^\ast$.
\end{lemma}

\begin{proof} The inclusions $\cD_0\subseteq \cD$ and $\cA\subseteq
\cD\cap\cD^\ast$ are elementary. If $a\in \cD$ and $x\in\cJ_0$ then
for some $n\geq 1$ we have
\[
\begin{array}{c}
ax=a(p_nx)=(ap_n)x\in \cJ.\cJ_0\subseteq \cJ.\\
\end{array}
\]
A density argument now implies that $a\cJ\subseteq\cJ$. By taking
adjoints we conclude that $\cD\cap\cD^\ast\subseteq\cA $.
\end{proof}

\Note In spite of the notation we do not claim that $\cD$ is the
norm closure of $\cD_0$.

\begin{lemma}\label{weaklim}
Let $\{p_n\}_{n=1}^\infty$ be an increasing sequence of projections
on $\cH$ that converge strongly to $1$ and let $\cJ$ and $\cA$ be
constructed as described in Theorem~\ref{Jstart}. If $\{
\phi_r\}_{r=1}^\infty$ is a sequence of unit vectors in $\cH$ and
$\lim_{r\to\infty}\norm p_n\phi_r\norm=0$ for every $n\geq 1$ then
$\lim_{r\to\infty}\norm a\phi_r\norm=0$ for every $a\in\cJ$.
\end{lemma}

\begin{proof} This is elementary if $a\in\cJ_0$ and follows for all
$a\in\cJ$ by approximation.
\end{proof}

We say that a sequence $\{\phi_r\}_{r=1}^\infty$ of unit vectors in
$\cH$ is localized (with respect to $\cJ$) if there exists $n\geq 1$
and $c>0$ such that $\norm p_n\phi_r\norm\geq c$ for all $r\geq 1$.

\begin{theorem}\label{localize}
If $x\in\cA$ and $\lam\in\sig(x)\backslash \sig(\pi_\cJ(x))$ then
there exists a sequence $\{\phi_r\}_{r=1}^\infty$ that is localized
with respect to $\cJ$ and satisfies either
\begin{equation}
\lim_{r\to\infty}\norm x\phi_r-\lam\phi_r\norm =0\label{localseq}
\end{equation}
or
\begin{equation}
\lim_{r\to\infty}\norm x^\ast\phi_r-\overline{\lam}\phi_r\norm
=0.\label{localseqst}
\end{equation}
\end{theorem}

\begin{proof} If $\lam\in\sig(x)$ then there exists a sequence
$\{\phi_r\}_{r=1}^\infty$ of unit vectors such that either
(\ref{localseq}) or (\ref{localseqst}) holds; see
\cite[Lemma~1.2.13]{LOTS}. Both cases are similar and we only
consider the first.

If $\lam\in\sig(x)\backslash \sig(\pi_\cJ(x))$ and (\ref{localseq})
holds and  $\lim_{r\to\infty}\norm p_n\phi_r\norm=0$ for all $n\geq
1$ then $\pi_\cJ (\lam 1-x)$ is invertible in $\cA/\cJ$, so there
exist $y\in\cA$ and $a\in\cJ$ such that
\[
y(\lam 1-x)=1+a.
\]
Lemma~\ref{weaklim} now yields
\begin{eqnarray*}
1&=&\lim_{r\to\infty}\norm (1+a)\phi_r\norm\\
&\leq&\lim_{r\to\infty}\left(\norm y\norm \,\norm (\lam 1-x)\phi_r\norm\right)\\
&=&0.
\end{eqnarray*}
The contradiction establishes that if $\lam\in\sig(x)\backslash
\sig(\pi_\cJ(x))$ then $\norm p_n\phi_r\norm$ does not converge to
$0$ as $r\to\infty$ for some $n\geq 1$. It follows that there exists
a subsequence $\{\psi_r\}_{r=1}^\infty$ and $c>0$ such that $\norm
p_n\psi_r\norm\geq c$ for all $r\geq 1$.
\end{proof}

\Note Theorem~\ref{localize} has no converse. If $a\in \cA$ is a
self-adjoint operator then any eigenvalue $\lam$ of $a$ that is
embedded in the continuous spectrum satisfies the conclusion of the
theorem for the choice $\cJ=\cK(\cH)$. One simply defines $\phi_n$
to be the normalized eigenvector of $a$ corresponding to the
eigenvalue $\lam$ for all $n$.

Sometimes one has several ideals in $\cA$ but neither is contained
in the other.

\begin{theorem}\label{J1J2} Let $\cJ_1$ and $\cJ_2$ be two ideals in the \ca\ $\cA$ with identity, and put $\cJ_3=\cJ_1\cap\cJ_2$. Then
\[
\sig_3(x)=\sig_1(x)\cup\sig_2(x)
\]
for all $x\in\cA$.
\end{theorem}

\begin{proof} It is elementary that $\cJ_3$ is an ideal. Let
$\cB=\cA/\cJ_1\oplus \cA/\cJ_2$ and define the $^\ast$-homomorphism
$\pi:\cA\to\cB$ by $\pi=\pi_1\oplus \pi_2$. Then the image
$\cC=\pi(\cA)$ is a \csa\ of $\cB$ and the kernel of $\pi$ is
$\cJ_3$. If $x\in\cA$ then the spectrum of $\pi (x)$ is the same
whether regarded as an element of $\cB$ or $\cC$. In the former case
the spectrum is $\sig_1(x)\cup\sig_2(x)$ and in the latter case it
is $\sig_3(x)$.
\end{proof}

We next describe one of the \ca s that we shall be using in the next
section. Let $\cH_1$ and $\cH_2$ be infinite-dimensional Hilbert
spaces and let $\cH=\cH_1\otimes \cH_2$ be their Hilbert space
tensor product. Let $I_i$ denote the identity operator on $\cH_i$
for $i=1,\, 2$.

\begin{theorem}\label{tensorJ}  Let $\{P_n\}_{n=1}^\infty$ be an increasing sequence of finite rank projections in $\cH_1$ which converges strongly to $I_1$ as $n\to\infty$ and put $p_n=P_n\otimes I_2$. Then $\cJ$ defined as in Theorem~\ref{Jstart} is the closed linear span of all operators $A_1\otimes A_2$ where $A_1\in\cK(\cH_1)$ and $A_2\in \cL(\cH_2)$. Also $\cA$, defined as in Theorem~\ref{Jstart}, contains the closed linear span of all operators $A_1\otimes A_2$ where $A_i\in\cL(\cH_i)$ for $i=1,\, 2$.
\end{theorem}

\begin{proof}  Let $\cJ'$ denote the closed linear span of all operators
$a=A_1\otimes A_2$ where $A_1\in\cK(\cH_1)$ and $A_2\in \cL(\cH_2)$.
The formula
\[
\lim_{n\to\infty}\norm A_1-P_n A_1P_n\norm =0
\]
implies
\[
\lim_{n\to\infty}\norm a-p_n ap_n\norm =0.
\]
We deduce that $a\in\cJ$ and hence that $\cJ'\subseteq \cJ$.
Conversely if $x\in\cJ_0$ then there exists $n\geq 1$ such that
$x=p_nxp_n$. If $P_n$ has rank $k$ then $p_nxp_n$ can be written as
the sum of $k^2$ terms of the form $A_1\otimes A_2$ where each $A_1$
has rank $1$. Hence $p_nxp_n\in\cJ'$. The inclusion $\cJ_0\subseteq
\cJ'$ implies $\cJ\subseteq \cJ'$. The final statement of the
theorem follows directly from the inclusions
\[
(A_1\otimes A_2)\cJ'\subseteq \cJ', \hspace{2em} %
\cJ' (A_1\otimes A_2)\subseteq \cJ'.
\]
\end{proof}

\section{Application to discrete Schr\"odinger operators\label{sectD}}

In this section we construct a \csa\ $\cA$ of $\cL(\cH)$ where
$\cH=l^2(\Z^d)$ by an ad hoc procedure. A more systematic approach
that uses a standard \ca\ is described in Section~\ref{sectC}.

We put $\cH_1=l^2(\Z)$ and $\cH_2=l^2(\Z^{d-1})$, so that
\begin{eq}
\cH\simeq \cH_1\otimes \cH_2\simeq l^2(\Z,\cH_2)\label{3tensor}
\end{eq}
by means of canonical unitary isomorphisms. We define the projections
$p_n$ by
\[
(p_n\phi)(x)=\begin{choices}
\phi(x)&\mbox{ if  } -n\leq x_1\leq n,\\
0&\mbox{ otherwise,}
\end{choices}
\]
for all $\phi\in \cH$ and $x\in \Z^d$. We also define the \ca\ $\cA$ and the ideal $\cJ$ as in Theorems~\ref{Jstart}~and~\ref{tensorJ}.
The ideal $\cJ$ contains all bounded operators on $\cH$ that are `concentrated' in some neighbourhood of the dislocation set $S=\{0\}\times \Z^{d-1}$. In Section~\ref{sectC} we explain how to generalize the ideas in this section by allowing the dislocation set to have a completely general shape.

\begin{lemma}
The \ca\ $\cA$ contains all `\Schrodinger operators' of the form
\begin{equation}
(A\phi)(x)=\sum_{r=1}^m a_r(x) \phi(x+b_r)\label{nsaschr}
\end{equation}
where $\phi\in l^2(\Z^d)$, $x\in\Z^d$, $m\in \Z_+$, $b_r\in \Z^d$
and $a_r\in l^\infty (\Z^d)$ for all $r\in \{ 1,2,\ldots,m\}$.
\end{lemma}

\begin{proof} An elementary calculation implies that $p_{n+k}Ap_n=Ap_n$ for
all $n\geq 1$ where $k=\max\{|b_r|:1\leq r\leq m\}$, so $A\in\cD_0$.
The same applies to $A^\ast$, so we may apply Lemma~\ref{Adef3}.
\end{proof}

We say that the \Schrodinger operator $A$ on $\cH$ is periodic in
the $\Z$ direction with period $k$ if $T_kA=AT_k$ where
$(T_k\phi)(m)=\phi(m+k)$ for all $\phi\in l^2(\Z,\cH_2)$. This holds
if and only if the coefficients $a_r,\, v$ are all periodic in the
the $\Z$ direction with period $k$.

\begin{theorem}\label{asymper}
If the \Schrodinger operator $A$ is periodic in the $\Z$ direction
with period $k$ then
\begin{equation}
\sig(\pi_\cJ(A))=\sigess(A)=\sig(A).\label{Aequal}
\end{equation}
If in addition $H=A+B+C$ where $B\in \cJ$ and $C\in \cK(\cH)$, then
\begin{equation}
\sigess(A)\subseteq \sigess(A+B)=\sigess(H)\subseteq
\sig(H).\label{Hcontain}
\end{equation}
\end{theorem}

\begin{proof} The identities in (\ref{Aequal}) follow directly from
Lemma~\ref{kideals} provided we can prove that $\sig(A)\subseteq
\sig(\pi_\cJ(A))$. If $\lam\in \sig(A)$ then there exists a sequence
$\{\phi_r\}_{r=1}^\infty$ of unit vectors such that either
$\lim_{r\to\infty} \norm A\phi_r-\lam \phi_r\norm =0$ or
$\lim_{r\to\infty} \norm A^\ast\phi_r-\overline{\lam} \phi_r\norm
=0$; see \cite[Lemma~1.2.13]{LOTS}. Both cases are similar, so we
only consider the first.

By translating the $\phi_r$ sufficiently and using the translation
invariance of $A$, we see that there exists a sequence
$\{\psi_r\}_{r=1}^\infty$ of unit vectors such that
$\lim_{r\to\infty} \norm A\psi_r-\lam \psi_r\norm =0$ and
$\lim_{r\to\infty}\norm p_n\psi_r\norm=0$ for every $n$. The
argument of Theorem~\ref{localize} establishes that $\lam\in
\sig(\pi_\cJ(A))$ and hence that $\sig(A)\subseteq
\sig(\pi_\cJ(A))$.

The statements in (\ref{Hcontain}) now follow from
Lemma~\ref{kideals} as soon as one observes that
$\sig(\pi_\cJ(H))=\sig(\pi_\cJ(A))$ and
$\sig(\pi_{\cK(\cH)}(H))=\sig(\pi_{\cK(\cH)}(A+B))$.
\end{proof}

The following theorem identifies the asymptotic part of the spectrum
of certain \Schrodinger operators $H$ as $x_1\to -\infty$. The
operators concerned have much in common with those of \cite{DS}, but
we allow them to be non-self-adjoint and require the underlying
space to be discrete.

\begin{theorem}\label{line} Let $S=\{x\in \Z^d:x_1\geq 0\}$ and put
\[
(p_n\phi)(x)=\left\{ \begin{array}{ll}
\phi(x)&\mbox{ if } x_1\geq -n,\\
0&\mbox{ otherwise,}
\end{array}
\right.
\]
for all $\phi\in l^2(\Z^d)$ and $n\geq 0$. Let $A$ be of the form
(\ref{nsaschr}) and suppose that it is periodic in the $x_1$
direction. Also let $H=A+B$ where $B$ is any bounded operator
confined to $S$ in the sense that $p_0B=Bp_0=B$. If $\cJ$ is defined
as in Theorem~\ref{Jstart} then
\[
\sig(A)=\sigess(A)=\sig(\pi_\cJ(H))\subseteq \sigess(H)\subseteq
\sig(H).
\]
\end{theorem}

We omit the proof, which is similar to that of Theorem~\ref{asymper}
and uses the fact that $B\in\cJ$.

We finally come to an application that involves two different closed
ideals. Let $H=A+V_1+V_2$ where $A$ acts on $\cH=l^2(\Z^2)$, is of
the form (\ref{nsaschr}) and is periodic in both horizontal and
vertical directions. We assume that the bounded potential $V_1$ has
support in $\Z\times [-a_2,a_2]$ while the bounded potential $V_2$
has support in $[-a_1,a_1]\times\Z $ for some finite $a_1,\, a_2$.
Let $\cJ_1$ be the ideal associated with the sequence of projections
\[
(p_n\phi)(i,j)=\left\{ \begin{array}{ll}
\phi(i,j)&\mbox{ if } -n\leq i\leq n,\\
0&\mbox{ otherwise,}
\end{array}
\right.
\]
and let $\cJ_2$ be the ideal associated with the sequence of
projections
\[
(q_n\phi)(i,j)=\left\{ \begin{array}{ll}
\phi(i,j)&\mbox{ if } -n\leq j\leq n,\\
0&\mbox{ otherwise.}
\end{array}
\right.
\]
The appropriate \ca\ $\cA$ is defined by
\[
\cA=\{ x\in\cL(\cH):x\cJ_1\subseteq \cJ_1, \, \cJ_1x\subseteq \cJ_1,
\,x\cJ_2\subseteq \cJ_2, \,\cJ_2x\subseteq \cJ_2\}.
\]
\begin{theorem}\label{cross} Under the above assumptions $H\in \cA$ and
\[
\sigess(H)=\sig_1(A+V_1)\cup\sig_2(A+V_2).
\]
If $V_1$ is periodic in the $x_1$ direction and $V_2$ is periodic in
the $x_2$ direction then
\[
\sigess(H)=\sig(A+V_1)\cup\sig(A+V_2).
\]
\end{theorem}

\begin{proof} Since $V_2\in\cJ_1$, we have $\sig_1(K)=\sig_1(H+V_1)$. Since
$V_1\in\cJ_2$, we have $\sig_2(K)=\sig_2(H+V_2)$. In order to apply
Theorem~\ref{J1J2} we need to prove that $\sig_3(K)=\sigess(K)$.
This follows if $\cJ_1\cap\cJ_2=\cK(\cH)$. The only non-trivial part
is to prove that if $x\in\cJ_1\cap\cJ_2$ then $x\in\cK(\cH)$.

Given such an $x$ put $x_{m,n}=p_m q_n x q_n p_m$ for all $m,\,
n\geq 1$. Noting that $p_m$ and $q_n$ commute and that their product
is of finite rank we see that $x_{m,n}\in\cK(\cH)$ for all $m,\, n$.
Since $x\in\cJ_2$ we have
\[
\lim_{n\to\infty}x_{m,n} =p_mxp_m
\]
and since $x\in\cJ_1$ we have
\[
\lim_{m\to\infty}p_mxp_m =x.
\]
Therefore $x\in\cK(\cH)$.

The final statement of the theorem involves an application of
Theorem~\ref{asymper}.
\end{proof}

\section{The standard \ca\ \label{sectC}}

If $\cA=\cL(\cH)$ for some infinite-dimensional, separable Hilbert space $\cH$ then $\cA$ contains only one non-trivial ideal, namely $\cK(\cH)$. In this section we construct a `slightly smaller' \ca\ which has a rich ideal structure. We formulate the theory in a general setting, even though our main applications are to $\Z^d$ and $\R^d$. However, exactly the same construction may be used for unbounded graphs and for waveguides, in which $X$ is an unbounded region in $\R^d$. In Section~\ref{hyperb} we show that it may also be applied to \Schrodinger operators on Riemannian manifolds, writing out the details in the case of three-dimensional hyperbolic space.

Let $(X,d,\mu)$ denote a space $X$ provided with a metric $d$ and a measure $\mu$; we require $X$ to be a complete separable metric space with infinite diameter in which every closed ball is compact; all balls in this paper are taken to have positive and finite radius. We also require that the measure of every open ball $B(a,r)=\{x\in X: d(x,a)<r\}$ is positive and finite. Let $\cU$ denote the class of all non-empty, open subsets of $X$.

If $S,T\subseteq X$ we put
\[
d(S,T)=\inf\{d(s,t):s\in S \mbox{ and } t\in T\}.
\]
The function $x\to d(x,S)$ is continuous on $X$; indeed
\[
|d(x,S)-d(y,S)|\leq d(x,y)\]
for all $x,\, y \in X$ and $S\subseteq X$. If $(X,d)$ is a length space in the sense of Gromov then
\[
\ol{B(a,r)}=\{ x\in X:d(x,a)\leq r\}
\]
and
\[
d(B(a,r),B(b,s))=\max \{ d(a,b)-r-s,0 \}.
\]
for all $a,\, b\in X$ and $r,\, s >0$. However, if $X=\Z^d$ with the Euclidean metric, neither of these identities need hold.

Now put $\cH=L^2(X,\mu)$. For any $S\in \cU$ we define the projection $P_S$ on $\cH$ by
\[
(P_S\phi)(x)=\begin{choices}
\phi(x)&\mbox{ if $x\in S$,}\\
0&\mbox{otherwise.}
\end{choices}
\]
We abbreviate $P_{B(a,r)}$ to $P_{a,r}$.

\begin{lemma}\label{lindelof}
If $A \in \cL(\cH)$ then there exists a largest open set $U$ such that $AP_U=0$. There also exists a largest open set $V$ such that $P_VA=0$.
\end{lemma}

\begin{proof} If $\cV$ is the class of all open sets $V$ such that $AP_V=0$ then the only candidate for $U$ is $U=\bigcup_{V\in\cV} V$ and by Lindel\"of's theorem we may also write $U=\bigcup_{n=1}^\infty V_n$ where $V_n$ is a sequence of sets in $\cV$. Put $W_1=V_1$ and $W_{n+1}=W_n\cup V_{n+1}$. If $W_n\in\cV$ then
\[
AP_{W_{n+1}}=AP_{W_n}+AP_{V_{n+1}}(1-P_{W_{n}})=0,
\]
so $W_{n+1}\in\cV$. It follows by induction that $AP_{W_n}=0$ for all $n\geq 1$. Now $P_{W_n}$ is an increasing sequence of projections that converges weakly to $P_U$ so $AP_U=0$. The second statement of the lemma has a similar proof.
\end{proof}

\begin{lemma}\label{epsball}
If $A,\, B \in \cL(\cH)$ and $AB\not= 0$ then for every $\eps >0$ there exists $a\in X$ such that $AP_{a,\eps}\not= 0$ and $P_{a,\eps}B\not= 0$.
\end{lemma}

\begin{proof} Let $\{a_n\}_{n=1}^\infty$ be a countable dense set in $X$ and define the sets $E_N$ inductively by
$E_1=B(a_1,\eps)$ and
\[
E_{n+1}=B(a_{n+1},\eps)\backslash (E_1\cup\ldots E_n).
\]
It follows directly that the sets $E_n$ are disjoint and that their union is $X$. Therefore
\[
\lim_{n\to\infty} \sum_{r=1}^n P_{E_r} =I,
\]
the limit being in the weak operator topology. Therefore
\[
\lim_{n\to\infty} \sum_{r=1}^n AP_{E_r}B =AB\not= 0
\]
in the same sense and there must exist $n$ such that $AP_{E_n}B\not= 0$. We conclude first that $AP_{E_n}\not= 0$ and $P_{E_n}B\not= 0$ and then that $AP_{a_n,\eps}\not= 0$ and $P_{a_n,\eps}B\not= 0$.
\end{proof}

We say that $A\in \cL(\cH)$ lies in $\cA_m$ if $P_{a,r}AP_{b,s}\not= 0$ implies $d(a,b)\leq r+s+m$. If $A$ has an integral kernel $K$ this amounts to requiring that $K(x,y)\not= 0$ implies $d(x,y)\leq m$, but we do not require that $A$ has such a kernel.

\begin{lemma}\label{Adeflemma} If $A\in \cA_m$ and $B\in \cA_n$ then $A^\ast\in \cA_m$, $A+B\in \cA_{\max(m,n)}$ and $AB\in \cA_{m+n}$.
\end{lemma}

\begin{proof} The invariance of $\cA_m$ under adjoints follows immediately from its definition.

If $P_{a,r}(A+B)P_{b,s}\not= 0$ then $P_{a,r}AP_{b,s}\not= 0$ or $P_{a,r}BP_{b,s}\not= 0$. Therefore $d(a,b)\leq r+s+m$ or $d(a,b)\leq r+s+n$. In both cases we deduce that $d(a,b)\leq r+s+\max(m,n)$.

If $P_{a,r}ABP_{b,s}\not= 0$ then Lemma~\ref{epsball} implies that for every $\eps >0$ there exists $c\in X$ such that $P_{a,r}AP_{c,\eps}\not= 0$ and $P_{c,\eps}BP_{b,s}\not= 0$. Therefore $d(a,c)\leq r+\eps +m$ and $d(c,b)\leq \eps+s+n$. These imply that $d(a,b)\leq r+s+m+n+2\eps$. Letting $\eps\to 0$ we finally deduce that $AB\in \cA_{m+n}$.
\end{proof}

We will frequently refer to the standard \ca\ $\cA$ below; this is defined in the next theorem.

\begin{theorem}\label{newAdef} If $\cA$ is the norm closure of $\ti{\cA}=\bigcup_{n=0}^\infty\cA_n$ then $\cA$ is a \csa\ of $\cL(\cH)$. If $V\in L^\infty(X,\mu)$ and $V$ also denotes the operator of multiplication by the function $V$, then $V\in \cA$. Moreover $\cK(\cH)\subseteq \cA$.
\end{theorem}

\begin{proof} The first statement follows directly from Lemma~\ref{Adeflemma}. If $P_{a,r}VP_{b,s}\not= 0$ then $P_{a,r}P_{b,s}V\not= 0$ and hence $P_{a,r}P_{b,s}\not= 0$. Therefore the open set $U=B(a,r)\cap B(b,s)$ is not empty, and there exists $c\in X$ with $d(a,c)<r$ and $d(b,c)<s$. Therefore $d(a,b)<r+s\leq r+s+0$ and $V\in \cA_0$.

If $A$ is compact and $A=AP_U=P_UA$ for some open set $U$ with diameter $n$ then $P_{a,r}AP_{b,s}\not= 0$ implies $P_{a,r}P_UAP_UP_{b,s}\not= 0$ and hence $P_{a,r}P_U\not= 0$ and $P_UP_{b,s}\not= 0$. Hence there exist $u,\, v\in U$ such that $d(a,u)<r$ and $d(v,b)<s$. We deduce that $d(a,b)<r+s+n$ so $A\in \cA_n$. Since the set of all such $A$ is norm dense in $\cK(\cH)$, we conclude that $\cK(\cH)\subseteq \cA$.
\end{proof}

The following alternative definition of $\cA$ is slightly more transparent in spite of the fact that it quantifies over a much larger class of sets.

\begin{theorem} Given $m\geq 1$, let $\cY_m$ denote the set of all $A\in \cL(\cH)$ such that for every $S\in \cU$ one has $AP_S=P_{S(m)}AP_S$. Then $\cA$ is the norm closure of $\bigcup_{m=0}^\infty\cY_m$.
\end{theorem}

\begin{proof} If we put $T(m)=X\backslash S(m)$ then $A\in \cY_m$ if and only if for every $S\in \cU$ one has $P_{T(m)}AP_S=0$.

Let $A\in \cA_m$, $0<r<1/3$ and $s=1/3$. If $S\in \cU$ and $b\in T(m+1)$ then $B(a,r)\subseteq S$ implies $d(a,b)\geq m+1> r+s+m$ and then $P_{b,s}AP_{a,r}=0$. Since $S$ may be written as the union of a countable number of balls $B(a,r)$ with $0<r<1/3$, Lemma~\ref{lindelof} implies that $P_{b,s}AP_S=0$. Since $T(m+1)$ may be covered by a countable number of balls $B(b,s)$, all with $s=1/3$, we deduce that $P_{T(m+1)}AP_S=0$. Therefore $A\in \cY_{m+1}$.

Conversely let $A\in \cY_m$, $r,\, s >0$ and $d(a,b)>r+s+m$. If we put $S=B(a,r)$ then $B(b,s)\subseteq T(m)$, so $P_{T(m)}AP_S=0$ implies $P_{b,s}AP_{a,r}=0$. Therefore $A\in \cA_m$.

The two inclusions together imply
\[
{\textstyle \bigcup_{m=1}^\infty\cA_m} ={\textstyle\bigcup_{m=1}^\infty\cY_m}
\]
and hence the statement of the theorem.
\end{proof}

We wish to associate an ideal $\cJ_S$ with every non-empty open subset $S$ of $X$. This may be done in two ways and we will prove that they yield the same result. The idea is to identify operators that `decrease in size' as one moves away from $S$.

If $S\in\cU$ and $r>0$ we put
\[
S(r)=\{x\in X:d(x,S)<r\}=
{\textstyle \bigcup \{ B(x,r):x\in S \} .}
\]
and
\begin{eqnarray*}
\cJ_{S,n}&=&\{ A\in \cA: A=P_{S(n)}AP_{S(n)} \} \\
         &=&\{ A\in \cA: A=AP_{S(n)}=P_{S(n)}A \} \\
         &=&\{ A\in \cA: 0=AP_{T(n)}=P_{T(n)}A \}
\end{eqnarray*}
where $T(n)=X\backslash S(n)=\{ x\in X: d(x,S)\geq n \}$. We also define
\begin{eqnarray*}
\cK_{S,n}&=&\{ A\in\cA: AP_{a,r}\not= 0 \implies d(a,S)\leq n+r \} \\
&&\cap\{A\in\cA: P_{a,r} A\not= 0 \implies d(a,S)\leq n+r \} \\
&=&\{ A\in\cA: d(a,S)> n+r  \implies  AP_{a,r}= P_{a,r}A= 0  \}.
\end{eqnarray*}

\begin{lemma}\label{JeqK}
If $n\geq 1$ then
\begin{eq}
{\textstyle\bigcup}\left\{ B(x,r):d(x,S)>r+n\right\} \subseteq T(n)\subseteq {\textstyle\bigcup}\left\{ B(x,r):d(x,S)>r+n-1\right\}.\label{BTinclusion}
\end{eq}
Hence
\begin{eq}
\cK_{S,n-1}\subseteq \cJ_{S,n}\subseteq \cK_{S,n}.\label{KJinclusion}
\end{eq}
\end{lemma}

\begin{proof}
If $y\in B(x,r)$ and $d(x,S)>r+n$ then $d(y,S)>n$. Hence $B(x,r)\subseteq T(n)$. This proves the first inclusion of (\ref{BTinclusion}).

If $x\in T(n)$ then $d(x,S)\geq n$. Putting $r=1/2$ we deduce that
$x\in B(x,r) $ and $d(x,S)>r+n-1$. This proves the second inclusion of (\ref{BTinclusion}).

If $A\in K_{S,n-1}$ then $AP_{x,r}=P_{x,r}A=0$ for all $x,\, r$ such that $d(x,S)>r+n-1$, so Lemma~\ref{lindelof} and the second inclusion of (\ref{BTinclusion}) together imply that $AP_{T(n)}=P_{T(n)}A=0$. Therefore $A\in \cJ_{S,n}$. On the other hand if $A\in \cJ_{S,n}$ then $AP_{T(n)}=P_{T(n)}A=0$. The first inclusion (\ref{BTinclusion}) of now implies that  $AP_{x,r}=P_{x,r}A=0$ whenever $d(x,S)>r+n$. Therefore $A\in \cK_{S,n}$. This completes the proof of (\ref{KJinclusion}).
\end{proof}

Let $\cF$ denote the family of all non-empty open sets $S$ such that $S(n)\not=X$
for every $n\geq 1$. We say that $S,\, T\in\cF$ are asymptotically
equivalent if for all $n\geq 1$ there exists $m\geq 1$ such that
$S(n)\subseteq T(m)$ and $T(n)\subseteq S(m)$. In particular all
nonempty, open, bounded sets are asymptotically equivalent to each
other.

\begin{theorem}\label{idealdef}
If $S\in \cU$ then
\[
{\textstyle\ol{\rule{0em}{2.2ex} \bigcup_{n=1}^\infty\cJ_{S,n}}} ={\textstyle\ol{\rule{0em}{2.2ex}\bigcup_{n=1}^\infty\cK_{S,n}}}.
\]
If $S\in \cF$ then this set, denoted by $\cJ_S$, is a proper, closed, two-sided ideal in $\cA$ and it contains $\cK(\cH)$. If $S,\, T$ are asymptotically equivalent than $\cJ_S=\cJ_T$.
\end{theorem}

\begin{proof} Lemma~\ref{JeqK} implies that
\[
{\textstyle \bigcup_{n=1}^\infty\cJ_{S,n}} ={\textstyle\bigcup_{n=1}^\infty\cK_{S,n}}.
\]
We denote this linear subspace of $\cA$ by $\ti{\cJ}_S$.

Let $A\in \cA_m$ and $B\in \cK_{S,n}$. If $ABP_{a,r}\not=0$ then $BP_{a,r}\not=0$ so $d(a,S)\leq n+r$. If $P_{a,r}AB\not=0$ then Lemma~\ref{epsball} implies that for every $\eps >0$ there exists $b\in X$ such that $P_{a,r}AP_{b,\eps}\not= 0$ and $P_{b,\eps}B\not= 0$. Therefore $d(a,b)\leq m+r+\eps$ and $d(b,S)\leq n+\eps$. We conclude that $d(a,S) \leq m+n+r+2\eps$. Since $\eps >0$ is arbitrary we deduce that $d(a,S)\leq m+n+r$. Therefore $AB\in \cK_{S,m+n}$. A similar argument can be applied to $BA$. These calculations imply that $\ti{\cA}\ti{\cJ}_S\subseteq \ti{\cJ}_S$ and $\ti{\cJ}_S\ti{\cA}\subseteq \ti{\cJ}_S$. The statement of the theorem now follows by a density argument.

In order to prove that $\cJ_S$ is proper we need to establish that
$\norm I-A\norm\geq 1$ for all $A\in\cJ_S$, or equivalently that
this holds for all $A\in\cJ_{S,n}$ and all $n\geq 1$. If $A=AP_{S(n)}=P_{S(n)}A$ then this follows from
\[
\norm I-P_{S(n)}\norm=\norm (I-P_{S(n)})(I-A)\norm\leq \norm I-P_{S(n)}\norm\, \norm
I-A\norm\leq \norm I-A\norm
\]
provided $\norm I-P_{S(n)}\norm = 1$. Since $S\in\cF$ there exists $a\in
X\backslash S(n+2)$. This implies that $B(a,1)\cap S(n)=\emptyset$.
Since $B(a,1)$ has positive measure there exists a non-zero $\phi\in
\cH$ whose support is contained in $B(a,1)$ and for which
$(I-P_{S(n)})\phi=\phi$.

If $A$ is a finite rank operator then $\lim_{n\to\infty}\norm
A-P_{S(n)}AP_{S(n)}\norm=0$, so $A\in\cJ_S$.
The same applies to all $A\in\cK(\cH)$ by a density argument.

If $S,\, T$ are asymptotically equivalent then routine algebra shows
that $\cJ_{S,0}=\cJ_{T,0}$. Once again a density argument implies
that $\cJ_{S}=\cJ_{T}$.
\end{proof}

If $A\in\cA$ and $S\in\cF$ we put $\sig_S(A)=\sig(\pi_{\cJ_S}(A))$.

\begin{theorem}
Let $S,\, T\in\cF$. If $S\subseteq T$ then $\cJ_S\subseteq \cJ_T$
and $\sig_S(A)\supseteq \sig_T(A)$ for every $A\in\cA$. If $S,\,
T\in\cF$ are asymptotically independent in the sense that
\begin{eq}
\forall n\geq 1.\, \exists m\geq 1.\, S(n)\cap T(n)\subseteq (S\cap
T)(m),\label{indep1}
\end{eq}
then
\begin{eq}
\cJ_{S\cap T}=\cJ_S\cap\cJ_T\label{indep2}
\end{eq}
and
\begin{eq}
\sig_{S\cap T}(A)=\sig_S(A)\cup\sig_T(A)\label{indep3}
\end{eq}
for every $A\in\cA$.
\end{theorem}

\begin{proof} If $A\in\cJ_{S,0}$ then there exists $n\geq 1$ such that
$A=AP_{S(n)}=P_{S(n)}A$. If $S\subseteq T$, this implies
$A=AP_{T(n)}=P_{T(n)}A$ and hence $\cJ_{S,0}\subseteq \cJ_{T,0}$.
Therefore  $\cJ_{S}\subseteq \cJ_{T}$ and $\sig_S(A)\supseteq
\sig_T(A)$ for every $A\in\cA$ by Lemma~\ref{kideals}.

If $S,\, T\in\cF$ we deduce that $\cJ_{S\cap
T}\subseteq\cJ_S\cap\cJ_T$. Now suppose that $S,\, T$ are
asymptotically independent and that $A\in \cJ_S\cap\cJ_T$. Equation
(\ref{J0closure}) implies
\[
\lim_{n\to\infty} \norm
A-P_{S(n)}AP_{S(n)}\norm=0,\hspace{2em}\lim_{n\to\infty} \norm
A-P_{T(n)}AP_{T(n)}\norm=0.
\]
If we put
\begin{eqnarray*}
A_n&=&P_{S(n)}P_{T(n)}AP_{T(n)}P_{S(n)}\\
&=&P_{S(n)\cap T(n)}AP_{T(n)\cap S(n)}
\end{eqnarray*}
then asymptotic independence implies
\[
A_n=P_{\{S\cap T\}(m)}A_n=A_nP_{\{S\cap T\}(m)}
\]
so $A_n\in\cJ_{S\cap T,0}$. Finally
\begin{eqnarray*}
\lim_{n\to\infty}\norm A-A_n\norm &\leq & \lim_{n\to\infty} \norm A-P_{S(n)}AP_{S(n)}\norm+%
\lim_{n\to\infty}\norm P_{S(n)}(A-P_{T(n)}AP_{T(n)})P_{S(n)}\norm\\
&\leq & \lim_{n\to\infty}\norm A-P_{S(n)}AP_{S(n)}\norm+%
\lim_{n\to\infty}\norm A-P_{T(n)}AP_{T(n)}\norm\\
&=& 0.
\end{eqnarray*}
Therefore $A\in\cJ_{S\cap T}$. Equation~(\ref{indep3}) finally
follows from Theorem~\ref{J1J2}.
\end{proof}

The \ca\ $\cA$ contains $L^\infty(X)$ and is therefore not
separable. It is unlikely that one can obtain a useful
classification of its irreducible representations, but a
\emph{partial} classification of its ideals can be obtained as
follows.

Let $\overline{X}$ be some compactification of $X$ and let $\partial
X=\overline{X}\backslash X$ denote the `points at infinity'. The
restriction of any $f\in C(\ol{X})$ to $X$ lies in
$L^\infty(X,\mu)$. Since every non-empty open subset of $X$ has positive measure we see
that
\begin{eq}
\norm f \norm_{C(\ol{X})}=\norm f \norm_{L^\infty}=\norm f \norm_{\cL(\cH)}%
=\norm f \norm_\cA.\label{4norms}
\end{eq}
It follows that $\cB=C(\overline{X})$ is a commutative \csa\ of
$\cA$. Note that there is a order-preserving one-one correspondence
between the ideals $\cI$ in $\cB$ and the open subsets $V$ of
$\overline{X}$. It is given by
\[
V_\cI=\{ x\in \overline{X}: f(x)\not= 0 \mbox{ for some } f\in \cI\}
\]
and
\[
\cI_V=\{f\in C(\ol{X}):f\vert_{\ol{X}\backslash V}=0\}.
\]

We will write $\ol{E}$ to denote the (compact) closure of a set
$E\sub \ol{X}$ in $\ol{X}$, even if $E\sub X$. If $U$ is an open
subset of $X$ then we define its set of asymptotic directions
$\ti{U}\sub \pd X$ to be the set of all $a\in\pd X$ that possess a
neighbourhood $V\sub\ol{X}$ for which $V\cap X\sub U$. It is
immediate that $\ti{U}$ is an open subset of $\pd X$ and that
$U\cup\ti{U}$ is an open subset of $\ol{X}$ with complement $\ol
{X\backslash U}$.

If $S\in\cU$ then $S(n)$ is an increasing sequence of open sets in
$X$, so $\ti{S(n)}$ is an increasing sequence of open subsets of
$\pd X$. We put
\[
\ha{S}=\bigcup_{n\geq 1}\ti{S(n)}
\]
and observe that $\ha{S}$ is also an open subset of $\pd X$.

\begin{example}
Let $ X=\R^d $  with the usual Euclidean metric and let $\Sigma$ be
the `sphere at infinity' parametrized by unit vectors $e$, called
directions.

\textbf{(a)} If
\[
S=\{ x\in X: \forall i\in\{ 1,2,\cdots,d \} . \, x_i> 0 \},
\]
then
\[
\ti{S(n)}=\ha{S}= \{ e\in \Sigma: \forall i\in\{1,2,\cdots,d\}. \,
e_i> 0\}
\]
for all $n\geq 1$. Therefore $\cB/(\cJ_S\cap\cB)\simeq C(K)$ where
\[
K=\{e\in \Sigma: \exists i\in\{1,2,\cdots,d\}. \, e_i\leq 0\}.
\]

\textbf{(b)} If we are only interested in asymptotics in a
particular direction $e\in\Sigma$ then we may define
\[
S=\R^d\backslash \bigcup_{r>0}\overline{B(re,r^{1/2})}.
\]
One sees that $S\in\cF$ and
\[
\ti{S(n)}=\ha{S}=\Sigma\backslash \{e\}
\]
for all $n\geq 1$. The quotient map $\pi$ from $\cB$ to
$\cB/(\cJ_S\cap\cB)\simeq \C$ is given by $\pi(f)=f(e)$.

\end{example}

\begin{lemma}\label{density}
If $S\in \cU$ then $\cJ_{S,0}\cap L^\infty(X,\mu)$ is dense in
$\cJ_S\cap L^\infty(X,\mu)$.
\end{lemma}

\begin{proof} Let $f\in J_S\cap L^\infty(X,\mu)$. If $p_n$ is the
multiplication operator associated with the characteristic function
of $S(n)$ then $p_nf\in \cJ_{S,0}\cap L^\infty$ for all $n\geq 1$
and $\lim_{n\to\infty}\norm f - p_nf\norm =0$ by (\ref{J0closure}).
\end{proof}

\begin{theorem}
The map $\cJ\to V_{\cJ\cap\cB}$ defines an order-preserving map from
ideals in $\cA$ to open subsets of $\ol{X}$. If $S\in\cU$ then
\[
V_{\cJ_{S}\cap\cB}=\ha{S}\cup X.
\]
If $S\in \cF$ then $\ha{S}\cup X\not= \ol{X}$.
\end{theorem}

\begin{proof}  The
first statement of the theorem depends on the observation that if
$\cJ$ is an ideal in $\cA$ then $\cJ\cap\cB$ is an ideal in $\cB$.

Given $S\in\cU$, we put $V=V_{\cJ_S\cap\cB}$. It follows directly from
the definitions that
\[
C_c(S(n)\cup \ti{S(n)})\sub \cJ_{S,0}\cap\cB\sub \cJ_S\cap\cB.
\]
Therefore $S(n)\cup \ti{S(n)}\sub V$ for all $n\geq 1$. Since $S$ is
non-empty, letting $n\to\infty$ we obtain $X\cup\ha{S}\sub V$.

If $a\notin X\cup\ha{S}$ then there exists $f\in C(\ol{X})$ such
that $f(a)=1$. Given $g\in \cJ_{S,0}\cap L^\infty(X)$ there exists
$n\geq 1$ such that $g=gp_n=p_ng$, where $p_n$ is the characteristic
function of $S(n)$. Since $a\in \ol{X\backslash S(n+2)}$, given
$\eps >0$, there exists $b\in X\backslash S(n+2)$ such that
$|f(b)-1| <\eps$. Putting $\eps=1/2$ there exists $\del\in (0,1)$
such that $x\in B(b,\del)$ implies $|f(x)|>1/2$ and $x\notin S(n)$.
The set $B(b,\del)$ has positive measure so $\norm f-g\norm_\infty
>1/2$. Lemma~\ref{density} now implies that $\norm f-h\norm_\infty
\geq 1/2$ for all $h\in \cJ_S\cap L^\infty(X)$ so $f\notin
\cJ_S\cap\cB$. Since this holds for all $f\in C(\ol{X})$ such that
$f(a)=1$ we conclude that $a\notin V$ and $V\sub X\cup\ha{S}$.

The final statement of the theorem follows from the fact that $S\in
\cF$ implies $1\notin \cJ_S$.
\end{proof}

\begin{corollary} If $S\in\cF$ then
\[
\cB/(\cJ_S\cap\cB)\simeq C(\pd{X}\backslash \ha{S}).
\]
\end{corollary}

\section{Pseudo-resolvents\label{pseudo}}

If one has a family of resolvent operators $R(z,A)$ all lying in a \ca\ $\cA$ and $\pi:\cA\to\cB$ is an algebra homomorphism with a non-trivial kernel $\cJ$, then $\pi(R(z))$ satisfy the resolvent equations in $\cB$. In this section we show how to define the spectrum of this new family, which is not the resolvent family of any obvious operator. This will be a crucial ingredient of our general theory.

Let $\cU$ denote the set of invertible elements of an associative
algebra $\cA$ with identity. If $a\in\cA$ the spectrum of $a$ is
defined by
\[
\sig(a)=\{\alp\in \C: \alp 1-a\notin\cU\}.
\]
If we put $U=\C\backslash \sig(a)$ and define $r:U\to \cA$ by
$r_z=(z1-a)^{-1}$ then $r$ satisfies the resolvent equations
\begin{eq}
r_\alp -r_\gam =(\gam -\alp )r_\alp r_\gam \label{psr}
\end{eq}
for all $\alp,\, \gam \in U$. Moreover
\[
1+(\gam-\alp)r_\alp=(\gam 1-a)r_\alp
\]
so $\sig(a)=\{ z:1+(z-\alp)r_\alp\notin \cU\}$.

Our goal in this section is to define the spectrum of a
pseudo-resolvent, defined as a function $r:U\to \cA$ that satisfies
(\ref{psr}) even though it is not generated by any $a\in \cA$.

If $A$ is a closed, unbounded operator on a Banach space $\cB$ and
$R(z,A)$ denotes its family of resolvent operators, defined for all
$z\notin \sig(A)$, then
\begin{equation}
\sig(R(z,A))=\{0\}\cup\{ (z-s)^{-1}:s\in\sig(A)\}\label{resolvspec}
\end{equation}
by \cite[Lemma 8.1.9]{LOTS}. This motivates our analysis, which is,
however, purely algebraic, making no reference to Banach spaces or
to unbounded operators. The advantage of this is that the results
are immediately applicable to quotient algebras $\cA/\cJ$, for which
no geometric interpretation exists.

\begin{theorem}\label{psrtheorem}
If $U\subseteq \C$ and $r:U\to\cA$ is a pseudo-resolvent and $\alp
\in  U$ then $1+(z-\alp)r_\alp\in \cU$ for all $z\in U$. If
\[
\ti{U}=\{ z: 1+(z-\alp)r_\alp\in \cU\}
\]
then $U\subseteq \ti{U}$ and the formula
\[
\ti{r}_z=r_\alp(1+(z-\alp)r_\alp)^{-1}
\]
defines an extension of the pseudo-resolvent from $U$ to $\ti{U}$.
Moreover $\ti{r}:\ti{U}\to\cA$ is a maximal pseudo-resolvent. The
set $\sig(r)=\C\backslash \ti{U}$ is called the spectrum of the
pseudo-resolvent $r$ and satisfies
\begin{eq}
\sig(r)=\{z:1+(z-\alp)r_\alp\notin \cU \}\label{rspec}
\end{eq}
for every choice of $\alp\in U$.
\end{theorem}

\begin{proof} By interchanging the labels $\alp ,\, \gam $ in (\ref{psr})
we see that $r_\alp $ and $r_\gam $ commute. Moreover
\begin{eqnarray*}
(1+(\gam -\alp )r_\alp )(1+(\alp -\gam )r_\gam )&=& 1+(\gam -\alp )\left\{ r_\alp -r_\gam -(\gam -a)r_\alp r_\gam \right\}\\
&=& 1,
\end{eqnarray*}
so both terms on the left hand side are invertible. This proves that
$U\subseteq \ti{U}$. If $\alp,\, z\in U$ then (\ref{psr}) implies
that
\[
r_\alp=r_z(1+(z-\alp)r_\alp )
\]
so $\ti{r}_z=r_z$ for all $z\in U$ and $\ti{r}$ is an extension of
$r$ to $\ti{U}$.

If $\bet ,\, \gam \in \ti{U}$ then
\begin{eqnarray*}
(\gam -\bet )r_\bet r_\gam &=& (\gam r-\bet r)r( 1+(\bet -\alp )r )^{-1}( 1+(\gam -\alp )r )^{-1}\\
&=& \left\{ ( 1+(\gam -\alp )r )-( 1+(\bet -\alp )r )\right\} r( 1+(\bet -\alp )r )^{-1}( 1+(\gam -\alp )r)^{-1}\\
&=& r\left\{ ( 1+(\bet -\alp )r )^{-1}- ( 1+(\gam -\alp )r)^{-1}\right\} \\
&=&r_\bet -r_\gam .
\end{eqnarray*}
Therefore $\ti{r}$ is a pseudo-resolvent on $\ti{U}$.

Now let $\ha{r}$ be a further extension of $\ti{r}$ to a
pseudo-resolvent on $\ha{U}\supseteq \ti{U}$. If $z\in \ha{U}$ then
by the first half of this proof $1+(z-\alp)r_\alp\in \cU$, so $z\in
\ti{U}$. Therefore $\ha{U}=\ti{U}$ and $\ti{r}:\ti{U}\to\cA$ is a
maximal pseudo-resolvent.

We have proved that $\ti{U}=\{z:1+(z-\alp)r_\alp\in \cU \}$ for all
$\alp\in \ti{U}$, and this proves (\ref{rspec}).
\end{proof}

\begin{corollary}\label{psrcorollary}
Let $\cJ$ be a two-sided ideal in the associative algebra $\cA$ with
identity and let $\pi:\cA\to \cA / \cJ$ be the quotient map. If
$U\sub \C$ and $r:U\to\cA$ is a maximal pseudo-resolvent then
\[
\sig(\pi(r))\subseteq \sig(r).
\]
\end{corollary}

\begin{proof} We need only observe that $z\to \pi(r_z)$ is a
pseudo-resolvent in $\cA/\cJ$ but its domain $U$ need not be
maximal. If its maximal extension has domain $V\supseteq U$ then
\[
\sig (\pi(r))=\C\backslash V\subseteq \C\backslash U =\sig(r).
\]
\end{proof}

\section{Perturbation theory\label{sectF}}

When extending the theory of Section~\ref{sectD} to differential
operators, one has to be careful not to refer to strong operator
convergence, because the standard \ca\ $\cA$ is only closed under
norm convergence. In this section we collect some of the technical
results that will be needed. These are formulated at the natural
level of generality, but the reader should keep in mind that they
will be applied to a resolvent operator $A$ acting in $L^2(\R^d,
\rmd x)$.

Let $X$ be a set with a countably generated $\sig$-field and a
$\sig$-finite measure $\mu$, and put $L^2=L^2(X, \mu)$.

\begin{lemma}\label{posA1}
Let $A$ be a linear operators $A$ on $L^2$ that is
positive in the sense that if $0\leq \phi\in L^2$ then $0\leq
A\phi\in L^2$. Then $A$ is bounded and
\[
\norm A\norm =\sup\{\norm A\phi\norm:0\leq \phi\in L^2 \mbox{ and }
\norm \phi\norm\leq 1\}<\infty.
\]
Moreover $|A(\phi)|\leq A(|\phi|)$ for all $\phi\in L^2$.
\end{lemma}

\begin{proof} See \cite[Lemma~13.1.1 and Theorem~13.1.2]{LOTS}.
\end{proof}

In the following discussion $V$ will always denote an unbounded
measurable function $V:X\to\C$, which we call a potential, and also
its associated multiplication operator. Given a positive operator $A$, let
$\tilde{\cV}_A$ denote the set of potentials $V$ that are relatively
bounded with respect to $A$ in the sense that
\[
\norm V\norm_A=\sup\{ \norm V(A\phi)\norm:\norm \phi\norm\leq 1\}
\]
is finite.

\begin{lemma}\label{posA2}
We have $\norm V\norm_A\leq \norm A\norm\,\norm V\norm_\infty$ for
all $V\in \ L^\infty$. Therefore $L^\infty(X)\subseteq
\tilde{\cV}_A$. If $|W|\leq |V|$ and $V\in\tilde{\cV}_A$ then
$W\in\tilde{\cV}_A$ and $\norm W\norm_A\leq\norm V\norm_A$. The
space $\tilde{\cV}_A$ is a Banach space with respect to the norm
$\norm \cdot\norm_A$.
\end{lemma}

\begin{proof} The last statement is the only one that is not elementary.
Let $\xi\in L^2$ satisfy $\norm \xi\norm_2 =1$ and $\xi(x)>0$ almost
everywhere in $X$ and let $\psi=A\xi$, so that $\psi\geq 0$. The
exists a measurable set $E$ such that $\psi(x)>0$ almost everywhere
in $E$ and $\psi(x)=0$ almost everywhere in $X\backslash E$. In many
cases $E=X$ but we do not assume this. If $\phi\in L^2$ and
\[
\phi_n(x)=\begin{choices}
\phi(x)&\mbox{if $|\phi(x)|\leq n\xi(x)$, }\\
\frac{n\xi(x)\phi(x)}{|\phi(x)|} &\mbox{ otherwise,}
\end{choices}
\]
then $|\phi_n|\leq |\phi|$ and $|\phi_n|\leq n\xi$. The dominated
convergence theorem implies that $\norm \phi_n-\phi\norm_2\to 0$ as
$n\to\infty$. Moreover
\[
|A(\phi_n)|\leq A(|\phi_n|)\leq A(n\xi)=n\psi
\]
so $A(\phi_n)$ has support in $E$. Letting $n\to\infty$ we conclude
that the same holds for $A(\phi)$. We conclude that if $V$ has
support in $X\backslash E$ then $ VA=0$, so we focus attention
henceforth on the restriction of all the potentials involved to $E$.

We next observe that $\norm V\psi\norm_2\leq \norm V\norm_A$ so if
$V_n$ is a Cauchy sequence in $\tilde{\cV}_A$ then $V_n\psi$ is a
Cauchy sequence in $L^2(E,\mu)$. Therefore $V_n\psi$ converges in
$L^2$ norm to a limit $V\psi$ in $L^2(E,\mu)$. There exists a
subsequence $n(r)$ such that  $V_{n(r)}$ converges almost everywhere
in $E$ to $V$.

Given $\eps>0$ there exists $N_\eps$ such that for all $m,\, n \geq
N_\eps$ we have
\[
\norm (V_m-V_n)(A\phi)\norm_2\leq \eps\norm\phi\norm_2
\]
for all $\phi\in L^2$. Replacing $n$ by $n(r)$, letting $r\to\infty$
and using Fatou's lemma we obtain
\[
\norm (V_m-V)(A\phi)\norm_2\leq \eps\norm\phi\norm_2.
\]
for all $m\geq N_\eps$ and all $\phi\in L^2$. Hence $V\in
\tilde{\cV}_A$ and $\norm V_m-V\norm_A\to 0$ as $r\to\infty$.
\end{proof}

Now let $\cV_A$ denote the closure of $L^\infty$ in $\tilde{\cV}_A$.

\begin{lemma}\label{posA3}
If $V\in \tilde{\cV}_A$ then $V\in \cV_A$ if and only if
$\lim_{n\to\infty}\norm V^{(n)}-V\norm_A=0$ where
\[
V^{(n)}(x)=\begin{choices}
V(x)&\mbox{if $|V(x)|\leq n$,}\\
\frac{nV(x)}{|V(x)|}&\mbox{otherwise.}
\end{choices}
\]
If $|W|\leq |V|$ and $V\in\cV_A$ then $W\in\cV_A$.
\end{lemma}

\begin{proof}  If $\norm W_n\norm_\infty\leq n$ and $\norm V-W_n\norm_A\to
0$ as $n\to\infty$ then by carrying out a separate calculation at
every $x\in X$ we see that
\[
|V-V^{(n)}|\leq |V-W_n|
\]
Lemma~\ref{posA2} now implies that
\[
\lim_{n\to\infty}\norm V-V^{(n)}\norm_A\leq \lim_{n\to\infty}\norm
V-W_n\norm_A=0.
\]
The second statement follows in a similar way from the inequality
\[
|W-W^{(n)}|\leq |V-V^{(n)}|.
\]
\end{proof}

\begin{lemma}\label{posA4} If $0\leq A \leq B$ as operators on $L^2$, in the sense that $0\leq A\phi\leq B\phi$ for all $\phi$ such that $0\leq \phi\in L^2$, then $\cV_B\subseteq \cV_A$.
\end{lemma}

\begin{proof} If $\phi\in L^2$ and $V\in \tilde{\cV}_B$ then
\[
|V(A\phi)|=|V|\, |A(\phi)|\leq |V|A(|\phi|)\leq |V|B(|\phi|)
\]
so
\[
\norm V(A\phi)\norm_2\leq \norm |V| (B |\phi|)\norm_2%
\leq \norm \,|V|\,\norm_B \norm \, |\phi|\,\norm_2 = \norm V\norm_B
\norm \phi\norm_2
\]
for all $\phi\in L^2$. This implies $\norm V\norm_A\leq \norm
V\norm_B<\infty$ and hence $V\in \tilde{\cV}_A$. The proof of the
lemma is completed as in Lemma~\ref{posA3}.
\end{proof}

We now specialize to the case in which $\cH=L^2(\R^d,\mu)$. Our goal
is to describe certain classes of potential in $\cV_A$, particularly
when $A$ is a positive convolution operator. Such operators arise as
the resolvents of constant coefficient, second order partial
differential operators and in certain other contexts; the reader
primarily interested in \Schrodinger operators should keep
Example~\ref{posA8} in mind.  We will use the classical $L^p$
inequalities due to H\"older, Young, Hausdorff-Young and
Riesz-Thorin without further mention.

\begin{lemma}\label{L1incA}
If $a\in L^1(\R^d)$ then the operator $A$ on $L^2(\R^d)$ defined by
$A\phi=a\ast\phi$ lies in the standard \ca\ $\cA$.
\end{lemma}

\begin{proof} If
\[
a_n(x)=\begin{choices}
a(x)&\mbox{if $|x|\leq n$,}\\
0&\mbox{otherwise}
\end{choices}
\]
and $A_n\phi=a_n\ast\phi$ then
\[
\lim_{n\to\infty}\norm A_n-A\norm\leq \norm a_n-a\norm_1 =0
\]
by Lemma~\ref{L1normbounds}. We combine this with the observation
that $A_n\in\cA_n$, because the support of $A_n\phi$ must lie within
a distance $n$ of the support of $\phi$.
\end{proof}

Let $\cC_d$ denote the set of operators $A$ on $L^2(\R^d,\rmd x)$
given by $A\phi=a\ast \phi$, where $0\leq a\in L^1(\R^d,\rmd x)$.

\begin{lemma}\label{posA5} If $A\in\cC_d$ and $a\in L^p$ for some $1<p\leq 2$ then $L^q\subseteq \cV_A$, where $1/p+1/q=1$.
\end{lemma}

\begin{proof} If $V\in L^q$ then
\[
\norm V(a\ast\phi)\norm_2 \leq \norm V\norm_q \norm
a\norm_p\norm\phi\norm_2,
\]
so
\[
\norm V\norm_A\leq \norm V\norm_q \norm a\norm_p.
\]
\end{proof}

\begin{lemma}\label{posA6} If $A\in\cC_d$ and $\ha{a}\in L^p$ where $\ha{a}$ denotes the Fourier transform of $a$ and $2\leq p<\infty$, then $L^p\subseteq \cV_A$.
\end{lemma}

\begin{proof} This uses the bound
\begin{eq}
\norm VA\norm\leq c_{d,p}\norm V\norm_p\norm
\ha{a}\norm_p.\label{2Lpnorms}
\end{eq}
See, for example, \cite[Theorem~5.7.3]{LOTS}.
\end{proof}

There are many other
results of a similar type in which both of the $L^p$ norms in
(\ref{2Lpnorms}) are replaced by other choices. See
\cite[Chapter~4]{traceideals} for details.

The following type of bound is used when analyzing multi-body
\Schrodinger operators. The decomposition of $\R^d$ used below may
be combined with a Euclidean rotation of $\R^d$, since this amounts
to a change of coordinate system.

\begin{theorem}\label{posA7} Let $x=(x_1,x_2)\in \R^{d_1}\times \R^{d_2}$ where $d=d_1+d_2$ and suppose that $|V(x_1,x_2)|\leq W(x_1)$ for all $x\in\R^d$ where $W\in L^p(\R^{d_1})$ and $2\leq p<\infty$. Suppose also that $A\in \cC_d$, $B\in \cC_{d_1}$ and
\[
|\ha{a}(\xi_1,\xi_2)|\leq |\ha{b}(\xi_1)|
\]
for all $\xi\in \R^{d}$, where $0\leq b\in L^1(\R^{d_1})$ and
$\ha{b}\in L^p(\R^{d_1})$. Then $V\in \cV_A$.
\end{theorem}

\begin{proof} We may write V=XW where $|X|\leq 1$. We may also write $A=BC$
where $\norm C\norm\leq 1$; in fact $C=\cF^{-1} D\cF$ where $\cF$ is
the Fourier transform and $D$ is the operator of multiplication by a
function $d$ with $|d|\leq 1$. Therefore
\[
\norm VA\norm=\norm XWBC\norm\leq \norm WB\norm\leq c\norm W\norm_p
\]
by applying Lemma~\ref{posA6} in $\R^{d_1}$.
\end{proof}

\begin{example}\label{posA8} If $H=-\ol{\lap}$ acting in $L^2(\R^d)$ with the usual domain then $A=(I+H)^{-1}$ is of the form $A\phi=a\ast \phi$ where $0\leq a\in L^1(\R^d)$ and $\ha{a}(\xi)=(1+|\xi|^2)^{-1}$ for all $\xi\in \R^d$. Theorem~\ref{posA7} is applicable in this context because
\[
(1+|\xi|^2)^{-1}\leq(1+|\xi_1|^2)^{-1}
\]
whenever $\xi=(\xi_1,\xi_2)$. One needs to assume that $p\geq 2$ and
$p>d_1/2$.
\end{example}

\section{Applications to differential operators\label{sectE}}

In this section we show that the \ca\ methods developed above can be
used to study the spectra of certain differential operators. Instead
of trying to study $\sig(A)$ directly we may redirect our attention
to the spectrum of one of its resolvent operators by virtue of the
results in Section~\ref{pseudo}. We say that the closed, unbounded
operator $A$ is \emph{affiliated to} the \csa\ $\cA$ of $\cL(\cH)$
if the conditions of the following lemma are satisfied.

\begin{lemma} Let $\cA$ be a \csa\ of $\cL(\cH)$ and let $R(z,A)\in\cA$ for some $z\notin\sig(A)$. Then
$R(w,A)\in\cA$ for all $w\notin\sig(A)$.
\end{lemma}

\begin{proof} If $X=I+(w-z)R(z,A)$ then $X\in\cA$ and
\begin{eqnarray*}
\sig(X)&=&\{1\}\cup\left\{ 1+\frac{w-z}{z-s}:s\in\sig(A)\right\}\\
&=& \{1\}\cup\left\{\frac{w-s}{z-s}:s\in\sig(A)\right\}.
\end{eqnarray*}
Since this does not contain $0$ we deduce that $X$ is invertible in
$\cL(\cH)$, and hence also invertible in $\cA$. Since
$R(w,A)=R(z,A)X^{-1}$ as in \cite[Theorem 1.2.10]{LOTS}, we deduce
that $R(w,A)\in\cA$.
\end{proof}

We say that a one-parameter group or semigroup $T_t$ is affiliated
to $\cA$ if its generator is affiliated in the above sense, i.\ e.\
if the associated resolvent family lies in $\cA$. If $H$ is a
typical \Schrodinger operator acting in $L^2(\R^d)$, then the
unitary operators $\rme^{-iHt}$ do not lie in the standard \ca\
$\cA$, but we will see they are affiliated to it.

Let $\cH= L^2(\R^d)$ and let $H_0$ be a constant coefficient
differential operator whose symbol is the polynomial $p$, so that
$H_0\phi=\cF^{-1}p\cF\phi$ where $\cF$ is the Fourier transform
operator and $p$ is regarded as an unbounded multiplication
operator. It is immediate that $H_0$ is a closed operator on
\[
\Dom(H_0)=\{\phi\in\cH:p\cF\phi\in\cH\}.
\]

\begin{theorem}\label{constdiff}
Suppose that $\lim_{|\xi|\to\infty}|p(\xi)|=+\infty$ and that there
exists a real constant $b$ such that $\Re(p(\xi))\leq b$ for all
$\xi\in\R^d$. Then $\sig(H_0)\subseteq \{z:\Re(z)\leq b\}$. If
$\Re(z)>b$ then $R(z,H_0)$ lies in the standard \ca\ $\cA$.
\end{theorem}

\begin{proof} We have $R(z,H_0)=\cF^{-1}\rho\cF$ where $\rho\in C_0(\R^d)$
is defined by
\[
\rho(\xi)=(z-p(\xi))^{-1}.
\]
If $n\geq 1$ we define
\[
\rho_n(\xi)=\rme^{-|\xi|^2/n}(z-p(\xi))^{-1}.
\]
Putting $R=R(z,H_0)$ and $R_n=\cF^{-1}\rho_n\cF$ we see that
\[
\lim_{n\to\infty}\norm R_n-R\norm=\lim_{n\to\infty}\norm
\rho_n-\rho\norm_\infty=0
\]
so it is enough to prove that $R_n\in\cA$ for all $n\geq 1$. Since
$\rho_n$ lies in the Schwartz space $\cS$ it is enough to observe
that $R_n\phi=k_n*\phi$ for all $\phi\in L^2$ where $k_n\in
\cS\subseteq L^1$; we may then apply Lemma~\ref{L1incA}.
\end{proof}

Before starting applications we change conventions so as to conform to the standard practice in quantum theory, writing $-H$ where one might expect to see $H$.

\begin{example} The differential operator
\[
(H_0\phi)(x,y) =-\frac{\partial^2\phi}{\partial x^2}-
\frac{\partial^3\phi}{\partial y^3}
\]
acting in $L^2(\R^2)$ has symbol $p(\xi,\eta)=\xi^2+i\eta^3$ and is highly non-elliptic. Nevertheless the conditions of Theorem~\ref{constdiff} are satisfied. The same applies to the non-negative, self-adjoint, differential operator acting in $L^2(\R^2)$ with real symbol
\[
p(\xi,\eta)=\xi^2+(\eta-\xi^n)^2
\]
where $n\geq 2$.
\end{example}

The following hypothesis is valid for a variety of second order
elliptic differential operators with variable coefficients; see
\cite{HKST}.
\begin{description}
\item[Hypothesis 1] The operator $-H_0$ is the generator of a strongly
continuous \ops\ $\rme^{-H_0t}$ on $L^2(\R^d)$. Moreover there
exist positive constants $c,\, \alp$ and an integral kernel
$K(t,x,y)$ such that
\begin{equation}
0\leq K(t,x,y)\leq c t^{-d/2}\rme^{-\alp|x-y|^2/t}\label{gaussian1}
\end{equation}
for all $t>0$ and $x,y\in\R^d$ and
\begin{equation}
(\rme^{-H_0t}\phi)(x)=\int_{\R^d}K(t,x,y) \phi(y)\, \rmd
y\label{gaussian2}
\end{equation}
for all $\phi\in L^2(\R^d)$ and $x\in\R^d$.
\end{description}

\begin{lemma} Under Hypothesis~1
\[
\sig(H_0)\subseteq \{ z:\Re(z)\geq 0\}
\]
and $(\lam I+H_0)^{-1}$ has an integral kernel $G(\lam ,x,y)$ for every $\lam >0$. There exists a function $g_\lam\in L^1(\R^d)$ and a constant $c_1>0$ such that
\[
0\leq G(\lam ,x,y)\leq g_\lam (x-y)
\]
and
\[
\norm (\lam I+H_0)^{-1}\norm \leq \norm g_\lam \norm_1 =c_1
\lam^{-1}<\infty. \]
\end{lemma}

\begin{proof} If we put
\[
k_t(x)=c t^{-d/2}\rme^{-\alp|x|^2/t}
\]
then there exists $c_1>0$ such that $\norm k_t\norm_1=c_1$ for all
$t>0$. Therefore $\norm\rme^{-H_0t}\norm\leq c_1$ for all $t>0$ and
$\sig(H_0)\subseteq \{z:\Re(z)\geq 0\}$. If $\lam >0$ the kernel $G$
satisfies
\begin{eqnarray*}
0\leq G(\lam ,x,y)&=&\int_0^\infty K(t,x,y)\rme^{-\lam t}\,\rmd t\\
&\leq & \int_0^\infty k_t(x-y)\rme^{-\lam t}\,\rmd t\\
&=& g_\lam(x-y),
\end{eqnarray*}
where the positivity of the functions involved implies that
\[
\norm g_\lam\norm_1=\int_0^\infty \norm k_t\norm_1\rme^{-\lam
t}\,\rmd t=c_1/\lam.
\]
Note finally that
\[
\ha{g}_\lam (\xi)=\frac{c_1}{\lam+c_2|\xi|^2}
\]
for some $c_2>0$, all $\lam >0$ and all $\xi\in\R^d$.
\end{proof}

\begin{example}\label{fractpower}
A bound of the type (\ref{gaussian1}) is not valid for
fractional powers of the Laplacian, i.\ e.\ $H_0=(-\ol{\lap})^\alp$
where $0<\alp<1$. However, in this case the \ops\ $\rme^{-H_0t}$ has the kernel
\[
K(t,x,y)= k_t(x-y)>0
\]
for all $t>0$, where $\norm k_t\norm_1=1$ and $
\ha{k}_t(\xi)=\rme^{-t|\xi|^{2\alp}} $ for all $t>0$ and
$\xi\in\R^d$. The construction of $k_t$ uses the theory of
fractional powers of generators of one-parameter semigroups; see
\cite[Chapter~9.11]{yoshida}. The resolvent operator $(\lam
I+H_0)^{-1}$ has the kernel
\[
G(\lam ,x,y)= g_\lam(x-y)>0
\]
for all $\lam >0$, where
\[
g_\lam (x)=\int_0^\infty k_t(x)\rme^{-\lam t}\rmd t >0.
\]
One deduces that $\norm g_\lam \norm_1=\lam^{-1}<\infty$ and
\[
\ha{g}_\lam(\xi)=\left(\lam+|\xi|^{2\alp}\right)^{-1}
\]
for all $\lam >0$ and $\xi\in\R^d$. The methods developed in this
paper still apply.
\end{example}

The above results allow us to reformulate our problem.

\begin{description}
\item[Hypothesis 2]
Let $K:\R^d\times \R^d\to\R$ and $k\in L^1(\R^d)$ satisfy
\[
0\leq K(x,y)\leq k(x-y)
\]
for all $x,\,y\in\R^d$. Let $R_0$ be the positive operator
associated with $K(x,y)$ and let $B$ be the positive operator
associated with $k(x-y)$, so that $0\leq R_0\leq B$.
Lemma~\ref{posA4} now implies that $\cV_B\subseteq \cV_{R_0}$.
\end{description}

If $R_0=(\lam I+H_0)^{-1}$ in the following theorem then $R=(\lam
I+H_0+V)^{-1}$ and the assumption $\norm V\norm_{R_0}<1$ states that
$V$ has relative bound less than $1$ with respect to $H_0$ in the
conventional language of perturbation theory.

\begin{lemma}\label{RinA}
Given Hypothesis~2, let the potential $V\in\cV_{R_0}$ satisfy $\norm
V\norm_{R_0}<1$ and put
\begin{equation}
R=R_0(I+VR_0)^{-1}=R_0\sum_{n=0}^\infty (-VR_0)^n.\label{Rseries}
\end{equation}
Then the operators $R_0$ and $R$ both lie in the standard \ca\
$\cA$.
\end{lemma}

\begin{proof} Given $\eps>0$ there exists $c\in\Z_+$ such that
$\int_{|x_1|>c}|k(x)|\, \rmd x<\eps$. If we put
\[
K_c(x,y)=\left\{ \begin{array}{ll}
K(x,y)&\mbox{ if $|x-y|\leq c$,}\\
0&\mbox{ otherwise,}
\end{array}\right.
\]
and define the operator $T$ on $\cH$ by
\[
(T_c\phi)(x)=\int_{\R^d} K_c(x,y)\phi(y)\, \rmd y
\]
then $\norm R_0-T_c\norm<\eps$ and
$T_cP_{S(n)}=P_{S(n+c)}T_cP_{S(n)}$ for every $S\in \cF$ and $n\geq
1$. hence $T_c\in\cD_0$ and $R_0\in\cD$. Applying the same argument
to $R_0^\ast$ yields $R_0\in\cA$ by virtue of Lemma~\ref{Adef3}. Defining $V^{(r)}$ as in Lemma~\ref{posA3}, the
identities $V^{(r)} P_{S(n)}=P_{S(n)} V^{(r)}$ for all $S,\, n$ and
$r$ imply that $V^{(r)}R_0\in\cA$. Hence $VR_0\in\cA$. The norm
convergence of the series in (\ref{Rseries}) now implies that
$R\in\cA$.
\end{proof}

We conclude with two applications to quantum theory. In the first we
consider with the \Schrodinger operator $H=H_0+V$ acting in $L^2(\R^d)$,
where $H_0=-\ol{\lap}$ and $V=W+X$ is a sum of possibly
complex-valued potentials satisfying the conditions specified below.
Passing to the resolvent operators we actually consider
$R_0=(aI+H_0)^{-1}$, $R_1=(aI+H_0+W)^{-1}$ and $R=(aI+H)^{-1}$,
where $a>0$ is large enough to ensure that all the inverses exist.

\begin{theorem} Suppose that $V$ and $W$ lie in the space $\cV_{R_0}$
defined just before Lemma~\ref{posA3} and that $\norm
V\norm_{R_0}<1$, $\norm W\norm_{R_0}<1$. Suppose that $W$ is
periodic in the $x_1$ direction. Let $S=\{x\in \R^d:|x_1|<1\}$, so
that $S(n)=\{x\in \R^d:|x_1|<n+1\}$ for all $n\geq 1$. Suppose that
$X$ has support in $S(c)$ for some $c\geq 1$. Finally define
the ideal $\cJ_S\subseteq\cA$ as in
Theorem~\ref{idealdef}. Then
\begin{equation}
\sig(H_0+W)=\sigess(H_0+W)=\sig_S(H_0+W)=\sig_S(H)\subseteq\sigess(H)\subseteq
\sig(H)\label{Hspec}
\end{equation}
where $\sig_S(A)=\sig(\pi_{\cJ_S}(A))$ for every $A\in\cA$.
\end{theorem}

\begin{proof} The operators $R_0,\, R_1$ and $R$ all lie in $\cA$ for large enough $a>0$ by
Lemma~\ref{RinA}. The equation (\ref{Hspec}) is equivalent, by definition, to \begin{equation}
\sig(R_1)=\sigess(R_1)=\sig_S(R_1)=\sig_S(R)\subseteq\sigess(R)\subseteq
\sig(R)\label{Rspec}
\end{equation}
The proof of the first two equalities in (\ref{Rspec})
uses the periodicity of $R_1$ in the $x_1$ direction as in the proof
of Theorem~\ref{asymper}. We next observe that
\begin{eqnarray*}
I+VR_0&=& I+WR_0+XR_0\\
&=& \left\{ I+XR_0(1+WR_0)^{-1}\right\} (I+WR_0)\\
&=& (I+XR_1)(I+WR_0).
\end{eqnarray*}
Since $I+VR_0$ and $I+WR_0$ are invertible, it follows that $I+XR_1$
is invertible. Therefore
\begin{eqnarray*}
R&=& R_0(I+VR_0)^{-1}\\
&=&R_0(I+WR_0)^{-1}(I+XR_1)^{-1}\\
&=& R_1(I+XR_1)^{-1}.
\end{eqnarray*}
Since $X\in\cV_{R_0}$ has support in $S(c)$ and $\cJ_S$ is an ideal we
can use Lemma~\ref{posA3} to deduce that $XR_1\in \cJ_S$. Therefore
\[
\pi_{\cJ_S}(R)=\pi_{\cJ_S}(R_1)\left\{
I+\pi_{\cJ_S}(XR_1)\right\}^{-1}=\pi_{\cJ_S}(R_1).
\]
The inclusions in (\ref{Rspec}) now follow by applying
Lemma~\ref{kideals}.\end{proof}

\begin{example}\label{mbso}
We next point out the relevance of the above results to multi-body
\Schrodinger operators. Let $\cH=L^2(\R^3\times \R^3)$ and put $x=(x_1,x_2)$ where $x_i\in \R^3$. Let $H_0=-\ol{\lap}$ and define
\[
H=H_0+V_1(x_1)+V_2(x_2)+V_3(x_1-x_2).
\]
where all three potentials lie in
$L^2(\R^3)+C_0(\R^3)$. By allowing $V_1,\, V_2$ and $V_3$ to be complex-valued we include in our analysis the non-self-adjoint \Schrodinger operators that arise in when discussing resonances via complex scaling. For suitable choices of $V_i$ this operator might be regarded as describing
two (spinless) electrons orbiting around a fixed nucleus (a simplified Helium
atom). Standard estimates imply that $V_i$ all have relative bound $0$ with respect to $H_0$ and that they all lie in $\cV_{R_0}$ with $\norm V_i\norm_{R_0}<1/3$ provided $R_0=(aI+H_0)^{-1}$ and $a>0$ is large enough. Lemma~\ref{RinA} implies that all of the relevant resolvent operators lie in the standard \ca\ $\cA$.

One can produce several asymptotic sets from $\{x:|x_1|<1\}$, $\{x:|x_2|<1\}$,
$\{x:|x_1-x_2|<1\}$, and we will concentrate on two of these. If one puts
\[
S=\{x:|x_1|<1\}\cup \{x:|x_2|<1\}\cup\{x:|x_1-x_2|<1\},
\]
it is evident that $S\in\cF$ and that $V_1+V_2+V_3\in\cJ_S$. Hence
\[
\sig_S(H)=\sig_S(H_0)=[0,\infty).
\]
This set relates to the states in which both particles move away to
infinity and they also separate from each other. On the other hand
if one puts
\[
T=\{x:|x_2|<1\}\cup\{x:|x_1-x_2|<1\},
\]
it is evident that $T\in\cF$ and that $V_2+V_3\in\cJ_T$. Hence
\[
\sig_T(H)=\sig_T(H_0+V_1)=\sig(H_0+V_1).
\]
This set relates to the states in which particle 2 moves away to
infinity and also separates from particle 1, which may or may not
stay close to the nucleus. If $A=-\ol{\lap}+V_1$ acting in $L^2(\R^3)$ then by taking Fourier transforms with respect to $x_2$ it is seen that
\[
\sig(H_0+V_1)=\sig(A)+[0,\infty)
\]
where $\sig(A)=[0,\infty)\cup \{ \lam_n\}$, where  $\lam_n$ are the possibly complex-valued discrete eigenvalues of the operator $A$.
\end{example}

\section{Hyperbolic space\label{hyperb}}

Let $(X,d,\mu)$ denote a complete non-compact Riemannian manifold
$X$ with bounded geometry, Riemannian metric $d$ (in the sense of
the triangle inequality) and Riemannian measure $\mu$. The
Laplace-Beltrami operator $H=-\lap$ on $L^2(X,\mu)$ is essentially
self-adjoint of $C_c^\infty(X)$ and the spectrum of its closure is
contained in $[0,\infty)$. The one-parameter semigroup
$\{\rme^{-Ht}\}_{t\geq 0}$ is associated with a positive $C^\infty$
heat kernel $K$ by
\[
(\rme^{-Ht}f)(x)=\int_X K(t,x,y)f(y)\, \mu(\rmd y)
\]
The kernel $K$ satisfies
\[
\int_XK(t,x,y)\,\mu(\rmd x)=1
\]
for all $x\in X$ and $t> 0$. We wish to show that $\rme^{-Ht}$ and
$(\lam I+H)^{-1}$ lie in the \ca\ $\cA$ for all $t,\, \lam>0$.
Rather than proving this under the weakest possible conditions,
we consider the hyperbolic space $\H^3$, in which all of the
expressions involved may be written down explicitly. The proof that
we given may be extended to $\H^d$ for arbitrary $d\geq 2$ with
minimal effort.

The geometry of hyperbolic space is well-studied; see
\cite[Section~4.6]{Rat} for the results listed below. In the upper
half space model $\H^n$ is the set $\{ x\in \R^n:x_n>0\} $ with the
local Riemannian metric
\[
\rmd s^2=\frac{\rmd^2x_1+\cdots \rmd^2x_n}{x_n^2}.
\]
The global metric $d$ is given by
\[
\cosh(d(x,y)) = 1+ \frac{|x-y|^2}{2x_ny_n}
\]
and the volume element is given by
\[
\mu(dx)=\frac{\rmd x_1\cdots \rmd x_n}{x_n^n}.
\]
The area of the unit sphere $S(x,r)$ of radius $r>0$ does not depend
on $x\in X$ and is given by
\[
\rho(r)=c_n\sinh^{n-1}(r)
\]
where $c_3=4\pi$. If $f:(0,\infty)\to \R$ is any positive,
measurable function then
\[
\int_X f(d(x,y))\, \mu(\rmd y)=\int_0^\infty f(r)\rho (r)\,\rmd r
\]
for all $x\in X$.

If $X=\H^n$, the spectrum of $H=-\ol{\lap}$ acting in $L^2(X,\mu)$
is equal to $[(n-1)^2/4,\infty)$, but the $L^p$ spectrum depends on
$p$; see \cite{DST}. The heat kernel may be written in the form
$K(t,x,y)=k_t(d(x,y))$, where for $n=3$ we have
\[
k_t(r)=(4\pi t)^{-n/2}\frac{r}{\sinh(r)}\rme^{-t-d(x,y)^2/4t}.
\]
See \cite{DGM}; see also \cite{DM} for relevant upper and lower
bounds when $n\not= 3$. One verifies directly that
\begin{eqnarray*}
\int_X K(t,x,y)\mu(\rmd y)&=&\int_0^\infty k_t(r)\rho(r)\,\rmd r \\
&=& \int_0^\infty (4\pi)^{-1/2}t^{-3/2} r\sinh(r) \rme^{-t-r^2/4t}\,\rmd r\\
&=& \int_{-\infty}^\infty (4\pi)^{-1/2}t^{-3/2} 2^{-1}r \rme^{r-t-r^2/4t}\,\rmd r\\
&=& \int_{-\infty}^\infty (4\pi)^{-1/2}t^{-3/2} 2^{-1}r \rme^{-(r-2t)^2/4t}\,\rmd r\\
&=& 1
\end{eqnarray*}
for all $t>0$.

If $\lam >0$ the operator $(H+\lam I)^{-1}$ has a Green function $G$
given explicitly by $G(\lam,x,y)=g_\lam(d(x,y))$, where
\[
g_\lam(r)=\int_0^\infty \rme^{-\lam t}k_t(r)\,\rmd t= \frac{
\rme^{-r\sqrt{\lam +1}}  }{ 4\pi\sinh(r) }.
\]
A direct calculation establishes that
\begin{eq}
\int_X G(\lam,x,y)\mu(\rmd y)=\int_0^\infty g_\lam(r)\rho(r)\,\rmd r
=1/\lam\label{gint}
\end{eq}
for all $\lam >0$. We will need the following lemma.

\begin{lemma}\label{L1normbounds}
If
\[
(Rf)(x)=\int_X r(x,y)\mu(\rmd y)
\]
for all $x\in X$ and $f\in L^2(X,\mu)$, then
\[
\norm R\norm_{L^2(X,\mu)}^2\leq \left\{ \sup_{x\in
X}\int_X|r(x,y)|\mu(\rmd y)\right\} \left\{ \sup_{x\in
X}\int_X|r(y,x)|\mu(\rmd y)\right\}.
\]
\end{lemma}

See \cite[Cor.~2.2.15]{LOTS} for the proof.

\begin{theorem} If $\lam >0$ and $t>0$ then $\rme^{-Ht}$ and $(H+\lam I)^{-1}$ both lie in the standard \ca\ $\cA$.
\end{theorem}

\begin{proof} The proof is almost the same in both cases so we only treat
the resolvent operators. We have $(\lam I+H)^{-1}=A_n+B_n$ where
\begin{eqnarray*}
(A_nf)(x)&=&\int_X a_n(x,y)f(y)\mu(\rmd y),\\
a_n(x,y)&=&\ti{a}_n(d(x,y)),\\
\ti{a}_n(r)&=&\begin{choices}
g_\lam(r)&\mbox{if $r\leq n$,}\\
0&\mbox{otherwise,}
\end{choices}
\end{eqnarray*}
and
\begin{eqnarray*}
(B_nf)(x)&=&\int_X b_n(x,y)f(y)\mu(\rmd y),\\
b_n(x,y)&=&\ti{b}_n(d(x,y)),\\
\ti{b}_n(r)&=&\begin{choices}
g_\lam(r)&\mbox{if $r> n$,}\\
0&\mbox{otherwise.}
\end{choices}
\end{eqnarray*}
It follows from its definition that $A_n\in\cA_n$ and from
(\ref{gint}) and Lemma~\ref{L1normbounds} that
$\lim_{n\to\infty}\norm B_n\norm =0$.
\end{proof}

\begin{example} The ideas in the second part of Section~\ref{sectC} can be applied in the setting of hyperbolic space. In the upper half space model the natural compactification has $\partial \H^n\sim (\R^{n-1}\times\{0\})\cup\{\infty\}$. If we put $S=\{x\in \H^n:0<x_n<1\}$ then $S(m)=\{x\in \H^n:0<x_n<\rme^m\}$. Moreover $\ti{S(m)}=\ha{S}=\R\times\{0\}$ for all $m\geq 1$. Therefore the quotient map $\pi:\cB\to \cB/(\cJ_S\cap\cB)\simeq\C$ is given by $\pi(f)=f(\infty)$.
\end{example}

{\bf Acknowledgements} I should like to thank A Pushnitski and J Weir for helpful comments on an earlier version of the paper.

Department of Mathematics\\
King's College\\
Strand\\
London WC2R 2LS

E.Brian.Davies@kcl.ac.uk


\vspace{4em}

\begin{thebibliography}{99}

\bibitem{HKST} Davies E B: Heat Kernels and Spectral Theory.
Cambridge Univ. Press, Cambridge, 1989.

\bibitem{LOTS} Davies E B:
Linear Operators and Their Spectra. Camb. Univ. Press, Cambridge,
2007.

\bibitem{EBDsc} Davies E B: On Enss' approach to scattering theory. Duke Math. J. 47 (1980) 171-185.

\bibitem{DM} Davies E B, Mandouvalos N: Heat kernel bounds on hyperbolic spaceand Kleinian groups. Proc. London Math. Soc 57 (1988) 182-208.

\bibitem{DS} Davies E B, Simon B: Scattering theory for
systems with different spatial asymptotics on the left and right.
Commun. Math. Phys. 63 (1978) 277-301.

\bibitem{DST} Davies E B, Simon B, Taylor M: $L^p$ spectral theory of Kleinian groups. J. Functional. Anal. 78 (1988) 116-136.

\bibitem{DGM} Debiard A, Gaveau B, Mazet E: Th\'eor\`emes de comparaison en g'eom\'etrie riemannienne. Publ. Kyoto Univ. 12 (1976) 391-425.

\bibitem{Dix} Dixmier J: Les
C*-alg\`ebres et leurs repr\'esentations, Gauthier-Villars, Paris,
1969.

\bibitem{enss} Enss V: Asymptotic completeness for quantum mechanical potential scattering. Commun. Math. Phys. 61 (1978) 285=291.

\bibitem{Hinchcliffe1} Hinchcliffe J: PhD thesis, King's
College London, 2006.

\bibitem{ped} Pedersen G K: \ca s and their Automorphism Groups. Academic Press, London, 1979.

\bibitem{Rat} Ratcliffe J G: Foundations of Hyperbolic Manifolds. Springer-Verlag, New York, 1994.

\bibitem{RS2} Reed M, Simon B: Methods of
Modern Mathematical Physics, Vol. 2. Academic Press.

\bibitem{RS4}
Reed M, Simon B: Methods of Modern Mathematical Physics, Vol. 4.
Academic Press.

\bibitem{traceideals} Simon B: Trace ideals and their applications.
Cambridge Univ. Press, Cambridge, 1979.

\bibitem{simsc} Simon B: Phase space analysis of simple scattering systems: extensions of some work of Enss. Duke Math. J. 46 (1979) 119-168.

\bibitem{yoshida} Yoshida K: Functional Analysis. Springer-Verlag,
Berlin, 1965.



\end{thebibliography}
\end{document}